\documentclass{amsart}
\usepackage{hyperref,url}

\theoremstyle{definition}

\theoremstyle{remark}

\begin{document}

\title[Iterated primitives of logarithmic powers]{Iterated primitives of logarithmic powers}

\author{Luis A. Medina}
\address{Departmento de Mathematicas,
Rutgers University, Piscataway, NJ 08854 and 
University of Puerto Rico,  San Juan, PR 00931}
\email{lmedina@math.rutgers.edu}

\author{Victor H. Moll}
\address{Department of Mathematics,
Tulane University, New Orleans, LA 70118}
\email{vhm@math.tulane.edu}

\author{Eric S. Rowland}
\address{Department of Mathematics,
Tulane University, New Orleans, LA 70118}
\email{erowland@tulane.edu}

\subjclass[2000]{Primary 26A09, Secondary 11A25}

\date{March 10, 2010}

\keywords{Iterated integrals, unimodality, valuations, von Mangoldt function}

\begin{abstract}
The evaluation of iterated primitives of powers of logarithms is
expressed in
closed form. The expressions contain polynomials with coefficients
given in terms of the harmonic numbers and their generalizations. 
The logconcavity of these polynomials is established.
\end{abstract}

\maketitle

\newcommand{\nn}{\nonumber}
\newcommand{\ba}{\begin{eqnarray}}
\newcommand{\ea}{\end{eqnarray}}
\newcommand{\realpart}{\mathop{\rm Re}\nolimits}
\newcommand{\imagpart}{\mathop{\rm Im}\nolimits}
\newcommand{\lcm}{\operatorname*{lcm}}

\newtheorem{Definition}{\bf Definition}[section]
\newtheorem{Thm}[Definition]{\bf Theorem}
\newtheorem{Example}[Definition]{\bf Example}
\newtheorem{Lem}[Definition]{\bf Lemma}
\newtheorem{Note}[Definition]{\bf Note}
\newtheorem{Cor}[Definition]{\bf Corollary}
\newtheorem{Prop}[Definition]{\bf Proposition}
\newtheorem{Conj}[Definition]{\bf Conjecture}
\newtheorem{Problem}[Definition]{\bf Problem}
\numberwithin{equation}{section}

\section{Introduction} \label{sec-intro}
\setcounter{equation}{0}

The search for closed forms of definite integrals 
has the extra appeal of connecting many diverse 
areas of mathematics. At the beginning of Calculus, the 
student is usually told of 
\begin{equation}
\int_{-\infty}^{\infty}  e^{-x^{2}} \, dx = \sqrt{\pi},
\label{normal}
\end{equation}
\noindent
that hints to a relation between exponential and 
trigonometric functions. The reader will find in 
Chapter $8$ of \cite{irrbook} a collection of 
different proofs of (\ref{normal}). These include the 
classical ones, presented in most textbooks, as well as
one by R. A. Kortram \cite{kortram1} relating this 
evaluation to the number of representations of a 
number as a sum of squares. 

The second author has begun the project of establishing all the entries in 
the classical table of integrals by I. S.
Gradshteyn and I. M. Ryzhik \cite{gr}. An example of the surprising 
connections encountered in this process is the subject of this note. 

Consider the sequence of functions $\{ f_{n}(x): \,  n \in \mathbb{N} \}$ 
defined by the iterated integrals
\begin{eqnarray}
f_{0}(x) & = & \ln(1+x) \label{seq-1} \\
f_{n}(x) & = & \int_{0}^{x} f_{n-1}(t) \, dt. \nonumber
\end{eqnarray}
\noindent
The first few examples are given by
\begin{eqnarray}
f_{1}(x) & = & -x + (1+x) \ln(1+x) \nonumber \\
f_{2}(x) & = & -\frac{x}{4}(3x+2) + \frac{1}{2}(1+x)^{2} \ln(1+x) \nonumber \\
f_{3}(x) & = & -\frac{x}{36}(11x^{2}+15x+6) + 
\frac{1}{6}(1+x)^{3} \ln(1+x) \nonumber \\
f_{4}(x) & = & -\frac{x}{288}(25x^{3}+52x^{2}+42x+12) + 
\frac{1}{24}(1+x)^{4} \ln(1+x).  \nonumber
\end{eqnarray}
\noindent
This data suggests that 
\begin{equation}
f_{n}(x) = -xA_{n}(x) + B_{n}(x) \ln(1+x) 
\end{equation}
\noindent
where $A_{n}$ and $B_{n}$ are polynomials. Moreover, the value 
\begin{equation}
B_{n}(x) = \frac{1}{n!}(1+x)^{n}
\label{exp-B}
\end{equation}
\noindent
can be guessed from the data above. 

In view of the fact that a closed form for the polynomials 
$A_{n}(x)$ seems harder to find, we begin by considering 
\begin{equation}
\alpha_{n} := \lcm \{ \text{Denominators of } A_{n}(x) \},
\end{equation}
the sequence of denominators in
the reduced form of $A_{n}(x)$. This 
sequence starts with
\begin{equation}
\{ 1, \, 1, \, 4, \, 36, \, 288, \, 7200, \, 43200, \, 2116800, \, 
33868800 \}.
\end{equation}
\noindent
A direct search in
Neil Sloane's \cite{sloane}
gives no information. On the other hand, the quotient
\begin{equation}
\beta_{n}:= \frac{\alpha_{n}}{n \alpha_{n-1}} 
\label{beta-def}
\end{equation}
\noindent
gives the values 
\begin{equation}
\{ 1, \, 2, \, 3, \, 2, \, 5, \, 1, \, 7, \, 2, \, 3, \, 1, \, 11, \, 1, \, 13 \}
\end{equation}
\noindent
and this is the sequence $A014963$ of \texttt{Sloane}. Namely
\begin{equation}
\label{beta-form}
\beta_{n} = \begin{cases}
        p \quad & \text{ if } n \text{ is a power of }p \\
        1 \quad & \text{ otherwise}. 
       \end{cases}
\end{equation}
\noindent
The sequence $\beta_{n}$ is the exponential of the von Mangoldt function
\begin{equation}
\Lambda_{n} = \begin{cases}
        \ln p \quad & \text{ if } n \text{ is a power of }p \\
        0 \quad & \text{ otherwise}, 
       \end{cases}
\end{equation}
\noindent
one of the basic functions in the theory of prime numbers \cite{hardy4}.  \\

In this note we provide an explicit formula for the polynomial $A_{n}(x)$ 
in terms of harmonic numbers. This 
establishes the form of $\beta_{n}$ discussed above. 
Section \ref{sec-extension} 
considers the iterated primitives of a power of $\ln(1+x)$. A sequence
of polynomials involving 
the generalized harmonic numbers is given. 

\section{The recurrences} \label{sec-recur}
\setcounter{equation}{0}

The first few values of $f_{n}(x)$ suggest the ansatz 
\begin{equation}
f_{n}(x) = -xA_{n}(x) + B_{n}(x) \ln(1+x), 
\end{equation}
\noindent
with  $A_{n}$ and $B_{n}$ polynomials in $x$.  This 
is replaced in the relation $f_{n}'(x) = f_{n-1}(x)$ to produce
\begin{eqnarray}
B_{n}'(x) & = & B_{n-1}(x) \label{rec-1} \\
xA_{n}'(x) + A_{n}(x) & = & xA_{n-1}(x) + \frac{B_{n}(x)}{1+x} \label{rec-2}.
\end{eqnarray}
\noindent
The expression for $B_{n}(x)$ in (\ref{exp-B}) is obtained directly from here.

\begin{Thm}
\label{thm-an}
The polynomial $A_{n}$ is given for $n \geq 0$ by
\begin{equation}
A_{n}(x) = \frac{1}{n!} \sum_{k=1}^{n} \binom{n}{k} \left( H_{n} - 
H_{n-k} \right) x^{k-1},
\label{A-def}
\end{equation}
\noindent
where $H_{n} = 1 + \tfrac{1}{2} + \tfrac{1}{3} + \cdots + \tfrac{1}{n}$ 
is the harmonic number.
\end{Thm}
\begin{proof}
The formula is proved by induction. Write 
\begin{equation}
A_{n}(x) = \sum_{k=1}^{n} a_{n,k}x^{k-1}
\end{equation}
\noindent
and use the recurrence (\ref{rec-2}) with the explicit expression for $B_{n}$
to conclude that $a_{n,1} = 1/n!$ and for $2 \leq k \leq n$
\begin{equation}
ka_{n,k} = a_{n-1,k-1} + \frac{1}{n!} \binom{n-1}{k-1}.
\label{rec-3}
\end{equation}
\noindent
This recurrence now gives (\ref{A-def}).
\end{proof}

\begin{Note}
Iterated integrals of $\ln(1-x)$ were considered by 
Mathar \cite[section~4.3]{mathar}, who obtained a similar recurrence.
\end{Note}

\begin{Thm}
The polynomials $A_{n}(x)$ are given by 
\begin{equation}
A_{n}(x) = \frac{1}{n!} \sum_{k=0}^{n} \binom{n}{k} x^{k} 
\sum_{m=1}^{n-k} \frac{(-x)^{m-1}}{m}. 
\end{equation}
\end{Thm}
\begin{proof}
The definition of $f_{n}(x)$ implies that
\begin{equation}
f_{n}(x)   =  -x^{n} \sum_{j=1}^{\infty} \frac{(-x)^{j}}{j(j+1) \cdots (j+n)}.
\end{equation}
\noindent
Using the partial fraction decomposition
\begin{equation}
\frac{1}{j(j+1) \cdots (j+n)} = \frac{1}{n!} \sum_{k=0}^{n} 
\frac{(-1)^{k}}{j+k} \binom{n}{k}
\end{equation}
\noindent
produces
\begin{eqnarray}
n!f_{n}(x) & = & -x^{n} \sum_{j=1}^{\infty} (-x)^{j} \sum_{k=0}^{n} 
\frac{(-1)^{k}}{j+k} \binom{n}{k} \nonumber \\
& = & -x^{n} \sum_{k=0}^{n} (-1)^{k} \binom{n}{k} (-x)^{-k} 
\sum_{j=1}^{\infty} \frac{(-x)^{j+k}}{j+k} \nonumber \\
& = & - \sum_{k=0}^{n} (-1)^{k} \binom{n}{k} x^{n-k} 
\sum_{m=k+1}^{\infty} \frac{(-x)^{m}}{m} \nonumber \\
& = & - \sum_{k=0}^{n} \binom{n}{k} x^{n-k} 
\left( - \ln(1+x) - 
\sum_{m=1}^{k} \frac{(-x)^{m}}{m} \right) \nonumber \\
& = & (1+x)^{n} \ln(1+x) + \sum_{k=0}^{n} \sum_{m=1}^{k} 
(-1)^{m} \binom{n}{k} \frac{x^{n+m-k}}{m} \nonumber \\ 
& = & (1+x)^{n} \ln(1+x) + \sum_{k=0}^{n} \binom{n}{k} x^{k} 
\sum_{m=1}^{n-k} 
\frac{(-x)^{m}}{m}. \nonumber
\end{eqnarray}
\noindent
Dividing this sum by $-x$ gives the result.
\end{proof}

Comparing the two expressions produced for the polynomials $A_{n}$ yields
the next identity.

\begin{Cor}
Let $n \in \mathbb{N}$. Then 
\begin{equation}
-\sum_{k=0}^{n} \binom{n}{k} x^{k} \sum_{m=1}^{n-k} \frac{(-x)^{m}}{m} = 
\sum_{k=1}^{n} \binom{n}{k} x^{k} \sum_{m=1}^{k} \frac{1}{m+n-k}.
\end{equation}
\end{Cor}

\section{The denominators of $A_{n}(x)$} \label{sec-deno}
\setcounter{equation}{0}

In this section we analyze the polynomial $A_{n}(x)$ and compute 
its denominator when written in reduced form. The proof employs some 
elementary number theory.

\begin{Thm}
For $n \geq 1$, the common denominator for $A_{n}(x)$ is given by
\begin{equation}
\alpha_{n} = n!  
\, \lcm(1, \, 2, \, \dots, n ).
\label{lcm-2}
\end{equation}
\end{Thm}

The proof of the theorem employs a preliminary divisibility result. 

\begin{Lem}
Let $p$ be a prime, $n \geq 1$, and $k \geq 0$ such that $p^{k} \leq n$.  Then the power of $p$ dividing the denominator of $\begin{displaystyle} \binom{n}{p^{k}} \left( H_{n}- H_{n-p^{k}} \right) \end{displaystyle}$ is $p^k$.
\end{Lem}

\begin{proof}
Observe that 
\begin{equation}
H_{n} - H_{n-p^{k}} = \sum_{i=0}^{p^{k}-1} \frac{1}{n-i}.
\label{sum-1}
\end{equation}
\noindent
In (\ref{sum-1}), the index $i$ ranges over a set of length $p^{k}$; 
therefore there is a unique index $i_{0}$ such that $p^{k}$ divides 
$n-i_{0}$, namely $i_0 = n - p^k \lfloor n/p^k \rfloor$. For $i \neq i_{0}$ 
write 
$n-i = p^{\alpha_{i}}\beta_{i}$  and 
$n-i_{0} = p^{\alpha}\beta$,
with $\alpha \geq k, \, \alpha_{i} \leq k-1$ and $p$ not dividing $\beta, 
\beta_{i}$. Therefore
\begin{eqnarray}
H_{n} - H_{n-p^{k}} =
\sum_{i \neq i_{0}} \frac{1}{p^{\alpha_{i}} \beta_{i}} 
+ \frac{1}{p^{\alpha} \beta} \nonumber 
\end{eqnarray}
\noindent
can be expressed as
\begin{equation}
H_{n} - H_{n-p^{k}}  =  
\frac{A}{p^{\gamma}B} + \frac{1}{p^{\alpha} \beta} \nonumber 
\end{equation}
\noindent 
for some $\gamma \leq k-1$ and $B$ not divisible by $p$. This gives
\begin{equation}
H_{n} - H_{n-p^{k}}  =  
 \frac{p^{\alpha - \gamma} A \beta + B}{p^{\alpha}B \beta}
\end{equation}
\noindent
and  it follows that the denominator of $H_{n}-H_{n-p^{k}}$ is divisible 
exactly by $p^{\alpha}$. 

We now determine the exponent of the highest power of $p$ dividing $\binom{n}{p^k}$, which according to Kummer's theorem \cite{kummer2} is the number of borrows involved in subtracting $p^k$ from $n$ in base $p$.  Let $n = n_l \cdots n_\alpha \cdots n_k \cdots n_0$ be the standard base-$p$ representation of $n$; then $n - i_0 = n_l \cdots n_\alpha 0 \cdots 0$.  Since $n_\alpha \neq 0$ and $n_j = 0$ for $k - 1 < j < \alpha$, there are $\alpha - k$ borrows when subtracting $p^k$ from $n$ in base $p$.  Therefore the power of $p$ dividing $\binom{n}{p^k}$ is $p^{\alpha - k}$, and the power of $p$ in $\binom{n}{p^k} (H_n - H_{n - p^k})$ is $p^{\alpha - k} \cdot p^{-\alpha} = p^{-k}$.
\end{proof}

\noindent
{\emph{Proof of Theorem}}.  Only the terms $1, 2, \dots, n$ appear as part of the denominators of 
coefficients of the polynomial $n! A_{n}(x)$. Thus, the common denominator is a 
divisor of $\lcm(1,2,\dots,n)$. The previous Lemma shows that every prime 
power $p^{k} \leq n$ appears. It follows that the 
denominator of $n! A_n(x)$ is
$\lcm(1,2,\dots,n)$.

Since every prime dividing $n!$ also divides $\lcm(1,2,\dots,n)$, there is
no cancellation with the numerator of $\sum_{j=1}^n \binom{n}{j} (H_n -
H_{n-j}) x^{j-1}$ when we divide by $n!$, so the denominator of $A_n(x)$
is $n! \lcm(1,2,\dots,n)$.

\begin{Cor}
The expression $\beta_{n}$ in (\ref{beta-def}) is now given by 
\begin{equation}
\beta_{n} = \frac{\lcm(1, 2, \dots, n)}
                 {\lcm(1, 2, \dots, n-1)}.
\label{lcm-1}
\end{equation}
\noindent
This gives (\ref{beta-form}).
\end{Cor}

\section{The polynomial $A_{n}$ is logconcave} \label{sec-logconcave}
\setcounter{equation}{0}

A sequence of coefficients $\{ a_{0}, \, a_{1}, \, \dots, a_{n} \}$ 
is {\emph{unimodal}} is there is an index 
$j_{*}$ such that $a_{0} \leq a_{1} \leq \cdots \leq a_{j^{*}}$ and 
$a_{j_{*}} \geq a_{j^{*}+1} \geq \cdots \geq a_{n}$.  A polynomial is 
called unimodal if its sequence of coefficients is unimodal.  
The polynomial
is called {\em{logconcave}} if its coefficients satisfy 
$a_{j}^{2} \geq a_{j-1}a_{j+1}$. An elementary argument shows that logconcavity
implies unimodality \cite{wilf1}. Unimodal and logconcave sequences appear
frequently in algebra and combinatorics. The reader will find in 
\cite{brenti1994, stanley1989a} a survey of these results.

The logconcavity of the polynomial $B_{n}(x)$ is elementary; it simply 
corresponds to that of the binomial coefficients. The argument 
for $A_{n}(x)$ is established next. 

\begin{Thm}
The polynomial $A_{n}(x)$ is logconcave.
\end{Thm}
\begin{proof}
The result is equivalent to the inequality
\begin{equation}
\binom{n}{j}^{2} \left( H_{n} - H_{n-j} \right)^{2} \geq 
\binom{n}{j-1} \binom{n}{j+1} \left( H_{n}-H_{n-j+1} \right) 
 \left( H_{n}-H_{n-j-1} \right).
\label{log-con1}
\end{equation}
\noindent
Introduce the notation
\begin{equation}
f_{n}(j) = H_{n}-H_{n-j} = \sum_{i=0}^{j-1} \frac{1}{n-i};
\end{equation}
\noindent
then (\ref{log-con1}) is equivalent to 
\begin{equation}
\frac{(n+1)(n-j+1)}{j} f_{n}^{2}(j) - f_{n}(j) + 1 \geq 0. 
\end{equation}
\noindent
This is a quadratic inequality with discriminant 
\begin{equation}
\frac{j(4n+5)-4(1+n)^{2}}{j} \leq 
\frac{n(4n+5)-4(1+n)^{2}}{j}  = - \frac{3n+4}{j} < 0,
\end{equation}
\noindent
that gives the logconcavity of $A_{n}$. 
\end{proof}

\begin{Note}
There are many instances where polynomials appearing in connection with 
the evaluation of integrals are logconcave. For instance, the polynomial
\begin{equation}
P_{m}(a) = \sum_{l=0}^{m} d_{l,m} a^{l} 
\label{polypm-def}
\end{equation}
\noindent
with
\begin{equation}
d_{l,m} = 2^{-2m} \sum_{k=l}^{m} 2^{k} \binom{2m-2k}{m-k} \binom{m+k}{m}
\binom{k}{l},
\label{def-dlm}
\end{equation}
\noindent 
appears in the formula
\begin{equation}
\int_{0}^{\infty} \frac{dx}{(x^{4}+2ax^{2} + 1)^{m+1}} = \frac{\pi}{2} 
\frac{P_{m}(a)}{[2(a+1)]^{m + \tfrac{1}{2}}}.
\end{equation}
\noindent 
The reader will find in \cite{amram} different proofs of these formulas. The 
unimodality of $P_{m}(a)$ was established in \cite{bomouni1} and its 
logconcavity, in \cite{kauers-paule}. A direct proof of this result 
also appears in \cite{chen2009a}.  \\

The logconcavity of a sequence $\{a_{j}: \, 1 \leq j \leq n  \}$ 
can be expressed in terms of the 
operator $\mathfrak{L}$ defined by $\mathfrak{L}( \{ a_{j} \} ) = 
\{ a_{j}^{2} - a_{j-1}a_{j+1} \}$, where $a_{j} = 0 $ if $j<0$ or $j > n$. 
Thus, a positive sequence $\mathbf{a}$ is logconcave if 
$\mathfrak{L}(\mathbf{a})$ is positive. A sequence is called $r$-logconcave if 
$\mathfrak{L}^{(s)}(\mathbf{a})$ is positive for $0 \leq s \leq r$. It is 
called {\emph{infinite logconcave}} if it is $r$-logconcave for any $r \in 
\mathbb{N}$. 
\end{Note}

The next conjecture is based on extensive symbolic calculations.

\begin{Conj}
For every $n \in \mathbb{N}$, the polynomial $A_{n}(x)$ is infinite logconcave.
\end{Conj}

The question of logconcavity of a sequence $\{ a_{k} \}$ 
is intimately connected with the location of the zeros of its generating
polynomial $P(x) = a_{0} + a_{1}x + \dots + a_{n}x^{n}$. Newton showed 
that if the zeros of $P$ are real and negative, then $\{ a_{k} \}$ is 
logconcave. It turns out that the zeros of the polynomial $P_{m}(a)$, 
generated by the sequence 
$\{ d_{l,m} \}$ in (\ref{def-dlm}), do not satisfy this condition. The 
quest for a proof of the logconcavity of 
$\{d_{l,m} \}$ lead one of the authors to study the operator $\mathfrak{L}$. 
Experimental data suggested that this sequence was infinite
logconcave. A similar conjecture on the binomial 
coefficients was proposed as a testing problem. 

The operator $\mathfrak{L}$ {\emph{does not}} 
preserve logconcavity. In order to bypass this difficulty, McNamara
and Sagan \cite{mcnamara1}
introduce a remarkable method that gives infinite logconcavity. Given 
$r \in \mathbb{R}$
a sequence $\{a_k\}$ is called $r$-factor logconcave if
$a_k^2 \geq r a_{k+1}a_{k-1}$. The 
next lemma appears in \cite{mcnamara1}.

\begin{Lem}
\label{rlog}
Let $\{a_k\}$ be a non-negative sequence and let $r_0 = (3 + \sqrt{5})/2$. 
Then $\{a_k\}$ being $r_0$-factor logconcave implies that 
$\mathfrak{L}(\{a_k\})$ is too. So in this case $\{a_k\}$ is infinite 
logconcave.
\end{Lem}

This result can be used to check that a sequence $a_{n,j}$ is infinite 
logconcave for a specific fixed value of $n$. It was used in 
\cite{mcnamara1} to show that the sequence of binomial coefficients $\{\binom{n}{k}: 
0 \leq k \leq n \}$ is infinite logconcave for $n \leq 1450$. Their 
approach is remarkably simple: calculate 
$\mathfrak{L}^i(a_{n,j})$ for $i$ up to some bound $M$.  If the 
sequences $\mathfrak{L}(a_{n,j})$, $\mathfrak{L}^2(a_{n,j}), \dots, 
\mathfrak{L}^{M}(a_{n,j})$ are non-negative and $\mathfrak{L}^M(a_{n,j})$ 
is $r_0$-factor logconcave, then Lemma \ref{rlog} implies 
that $\{a_{n,j}\}_{j=0}^n$ is infinite logconcave.  Employing this 
technique we have verified
the following result.

\begin{Thm}
The polynomial $A_n(x)$ is infinite logconcave for all $n\leq 300$.
\end{Thm}

Extending the conjecture on the infinite logconcavity of the 
binomial coefficients, Stanley \cite{stanley2008}, McNamara-Sagan 
\cite{mcnamara1} and Fisk \cite{fisk-1} conjectured the next result. This 
was established by Petter Br\"{a}nd\'{e}n in \cite{branden1}.

\begin{Thm}
Suppose that the polynomial $\begin{displaystyle} \sum_{k=0}^{n} a_{k}x^{k}
\end{displaystyle}$ has only real and negative zeros. Then so does the 
polynomial
\begin{equation}
\sum_{k=0}^{n} \left( a_{k}^{2} - a_{k-1}a_{k+1} \right)z^{k}, 
\quad \text{ where }a_{-1} = a_{n+1} = 0. 
\end{equation}
\noindent
In particular, the sequence $\{a_{k}: \, 0 \leq k \leq n \}$ is infinite
logconcave.
\end{Thm}

The theorem established the fact that the binomial coefficients 
$\{ \binom{n}{k}: \, 0 \leq k \leq n \}$ are infinite 
logconcave for every $n \in \mathbb{N}$.  On the other hand, examples of
sequences conjectured to be infinite logconcave have 
generating polynomials with non-real zeros. This holds for $P_{m}(a)$ in 
(\ref{polypm-def}) as well as the polynomial $A_{n}(x)$ considered here.

\section{An extension} \label{sec-extension}
\setcounter{equation}{0}

The expression for the iterated integrals described here extends to 
iterated integrals obtained by 
\begin{eqnarray}
f_{0,j}(x) & = & \ln^{j}(1+x) \label{seq-j} \\
f_{n,j}(x) & = & \int_{0}^{x} f_{n-1,j}(t) \, dt. \nonumber
\end{eqnarray}

\noindent
Symbolic computation gives
\begin{eqnarray}
f_{1,2}(x) & = & 2x - 2(x+1)u + (x+1)u^{2} \nonumber \\
f_{2,2}(x) & = & \frac{1}{4}x(7x+6) - \frac{3}{2}(x+1)^{2}u + 
\frac{1}{2}(x+1)^{2}u^{2} \nonumber \\
f_{3,2}(x) & = & \frac{1}{108}x(85x^{2}+147x+66) - 
\frac{11}{18}(x+1)^{3}u + \frac{1}{6}(x+1)^{3}u^{2} \nonumber 
\nonumber
\end{eqnarray}
where $u = \ln(1+x)$.  The structure of $f_{n,j}(x)$ is provided below.

\begin{Thm}
\label{form-1}
For $n \geq 0$ and $j \geq 1$, there are polynomials $A_{n,j}(x)$ and $B_{n,k,j}(x)$ such that 
\begin{equation}
f_{n,j}(x) = A_{n,j}(x) + \sum_{k=1}^{j} B_{n,k,j}(x) \ln^{k}(1+x). 
\label{form-f}
\end{equation}
\noindent
Moreover, $B_{n,k,j}(x) = b_{n,k,j}(1+x)^{n}$, for some $b_{n,k,j} \in 
\mathbb{Q}$.
\end{Thm}
\begin{proof}
In the first step, it is shown that $f_{n,j}$ has the stated form, for some 
polynomials $A_{n,j}, B_{n,k,j}$. This is elementary: integration 
by parts shows that, for $r, \, k \in \mathbb{N}$,
\begin{equation}
\int s^{r} \ln^{k}s \, ds = P( \ln s) s^{r+1},
\end{equation}
\noindent
where $P$ is a polynomial of degree $k$. Induction on $n$ and
(\ref{seq-j}), give the result.

The next is to prove that $B_{n,k,j}(x)$ is a multiple of $(1+x)^{n}$. 
In order to achieve this, replace (\ref{form-f}) in (\ref{seq-j}) to 
obtain
\begin{eqnarray}
(1+x) A_{n,j}'(x) + B_{n,1,j}(x)  & = & (1+x) A_{n-1,j}(x) \label{eqn-a} \\
(1+x)B_{n,k,j}'(x) + (k+1)B_{n,k+1,j}(x) & = & 
(1+x) B_{n-1,k,j}(x) \quad \text{ for } 1 \leq k \leq j-1, \nonumber \\
B_{n,j,j}'(x) & = & B_{n-1,j,j}(x). \nonumber
\end{eqnarray}

Fix $n$ and $j$ and use induction on $k$ to show $B_{n,k,j}(-1) = 0$. 
The first equation in (\ref{eqn-a}) gives $B_{n,1,j}(-1)=0$ and the 
second one gives $B_{n,k,j}(-1) = 0$, using induction on $k$. The 
fact that $B_{n,k,j}(x)$ is a power 
of $(x+1)^{n}$ now follows directly from the recurrence (\ref{eqn-a}). Indeed,
\begin{equation}
B_{n,k+1,j}(x) = \frac{1+x}{1+k}B_{n-1,k,j}(x) - 
\frac{1+x}{1+k}B_{n,k,j}'(x)
\end{equation}
\noindent
gives the result for $1 \leq k \leq j-1$. The case $k=j$ is settled with 
$B_{n,j,j}'(x) =  B_{n-1,j,j}(x)$, where the vanishing of $B_{n,j,j}(x)$ at
$x=-1$ adjusts the constant of integration. Therefore
\begin{equation}
B_{n,k,j}(x) = b_{n,k,j}(1+x)^{n}
\end{equation}
\noindent
for some $b_{n,k,j}$ to be determined. 

For $n \geq 1$ and $j \geq 1$, the coefficients $b_{n,k,j}$ satisfy the equations
\begin{eqnarray}\label{b recurrence}
b_{0,j,j} & = & 1 \label{recu-b} \\
b_{0,k,j} & = & 0  \text{ for } 0 \leq k \leq j-1 \nonumber \\
b_{n,j,j} & = & \frac{1}{n} b_{n-1,j,j} \nonumber \\
b_{n,k,j} & = & -\frac{k + 1}{n} b_{n,k+1,j} + \frac{1}{n} b_{n-1,k,j}  \text{ for }  0 \leq k \leq j-1. \nonumber
\end{eqnarray}

The third recurrence in (\ref{recu-b}) can be solved directly to obtain
$b_{n,j,j} = 1/n!$. 
\end{proof}

To write $b_{n,k,j}$ more explicitly, let
\begin{equation}
	H_{n,m} = \sum_{1 \leq a_1 \leq a_2 \leq \cdots \leq a_m \leq n} 
\frac{1}{a_1 a_2 \cdots a_m}
\end{equation}
\noindent
for $n \geq 0$ and $m \geq 0$.  In other words, the sum is over all 
nondecreasing $m$-tuples with entries from $\{1, 2, \dots, n\}$.  $H_{n,m}$ 
generalizes the harmonic number $H_n = H_{n,1}$.  For $m = 0$ the product 
is empty, so $H_{n,0} = 1$ for $n \geq 0$.  For $n = 0$ the sum is empty 
unless $m = 0$, so $H_{0,m} = 0$ for $m \geq 1$.

For $n \geq 1$ and $m \geq 1$, we have 
$H_{n,m} = \frac{1}{n} H_{n,m-1} + H_{n-1,m}$ by breaking up the $m$-tuples 
in $H_{n,m}$ according to whether $a_m = n$ or not.  It follows that the 
proposed expression for $b_{n,k,j}$ satisfies
\begin{equation}
	b_{n,k,j} = -\frac{k + 1}{n} b_{n,k+1,j} + \frac{1}{n} b_{n-1,k,j}.
\end{equation}
\noindent
It follows from here that
\begin{equation}\label{bnkj}
b_{n,k,j} = \frac{(-1)^{j-k} j!}{n! \, k!} H_{n,j-k} 
\end{equation}
\noindent
satisfies the equations in  \eqref{b recurrence}. 

\begin{Thm}
\label{thm-formb}
The polynomial $B_{n,k,j}$ in Theorem \ref{form-1} is given by 
\begin{equation}
B_{n,k,j}(x) =  \frac{(-1)^{j-k} j!}{n!k!}
H_{n,j-k} (1+x)^{n}.
\end{equation}
\end{Thm}

\medskip

\begin{Note}
The recursion $H_{n,m}=\tfrac{1}{n}H_{n,m-1}+H_{n-1,m}$, with $n$ fixed, can 
be solved explicitly. The initial term is $H_{1,m}=1$. It may be checked 
directly that 
\begin{equation}
H_{n,m} = \sum_{k=1}^{n} (-1)^{k+1} \binom{n}{k} k^{-m}.
\end{equation}
\end{Note}

For $n=2$, $H_{2,j-k}$ is a geometric series; therefore (\ref{bnkj}) gives
\begin{equation}
	b_{2,k,j} = \frac{(-1)^{j-k} j! \, (2^{j-k+1} - 1)}{k! \, 2^{j-k+1}}.
\end{equation}

\medskip

The last step in the determination of $f_{n,j}$ is the 
construction of the polynomial $A_{n,j}(x)$ in (\ref{form-f}). \\

The equation for $A_{n,j}(x)$ comes from (\ref{eqn-a}): 
\begin{equation}
A_{n,j}'(x) = A_{n-1,j}(x) + \frac{(-1)^{j} j!}{n!}H_{n,j-1} 
(1+x)^{n-1},
\label{recu-a1}
\end{equation}
\noindent
with the boundary conditions $A_{0,j}(x) = 0$ and $A_{n,j}(0) = 0$.

The polynomial $A_{n,j}(x)$ is now written as 
\begin{equation}
A_{n,j}(x) = \sum_{r=1}^{n} c_{n,r,j}x^{r}
\end{equation}
\noindent
for some coefficients $c_{n,r,j}$. The fact that $A_{n,j}$ is of 
degree $n$ comes directly from (\ref{recu-a1}). This recurrence also 
yields
\begin{eqnarray}
c_{n,1,j} & = & \alpha_{n,j} \label{one-1} \\
c_{n,r,j} & = & \frac{1}{r}c_{n-1,r-1,j} + \frac{1}{r} \binom{n-1}{r-1} \alpha_{n,j}, \label{two-1}
\end{eqnarray}
\noindent
where 
\begin{equation}
\alpha_{n,j} = \frac{(-1)^{j} j! H_{n,j-1}}{n!}.
\end{equation}

The claim is that the solution of this recurrence yields the 
following expression for the polynomial $A_{n,j}(x)$. \\

\begin{Thm}
The polynomial $A_{n,j}(x)$ in Theorem \ref{form-1} is given by
\begin{equation}
A_{n,j}(x)  = \frac{(-1)^{j} j!}{n!} \sum_{r=1}^{n} 
\binom{n}{r} \left[  \sum_{k=0}^{r-1} 
\frac{H_{n-k,j-1}}{n-k} \right] x^{r}.
\label{thm-anj}
\end{equation}
\end{Thm}
\begin{proof}
Direct replacement of 
\begin{equation}
c_{n,r,j} = \frac{(-1)^{j} j!}{n!} \binom{n}{r} 
\sum_{k=0}^{r-1} \frac{H_{n-k,j-1}}{n-k}
\end{equation}
\noindent
in (\ref{one-1}) and (\ref{two-1}).
\end{proof}

\medskip

\begin{Note}
The case $j=1$ of (\ref{thm-anj}) reduces to the expression for $-xA_{n}(x)$
given in Theorem \ref{thm-an}. 
\end{Note}

\medskip

The final statement deals with an arithmetical property of the 
coefficients of  $B_{n,j}(x)$. This was first observed symbolically 
and it was one of the motivations for the work discussed here.

\begin{Cor}
Assume $p$ is a prime that divides the denominator of a
coefficient of $B_{n,k,j}(x)$ in Theorem \ref{thm-formb}. Then
$p \leq \max(n,k)$.
\end{Cor}

\noindent
{\bf Acknowledgments}. The authors wish to thank the referees for several 
suggestions and a 
careful reading of an earlier version of the paper. The second
author was partially funded by
$\text{NSF-DMS } 0070567$. The work of the third author was partially 
supported by Tulane VIGRE Grant $0239996$. 

\bigskip


\end{document}